\documentclass[12pt]{amsart}
\usepackage{fourier}
\usepackage[T1]{fontenc}
\usepackage{graphics}
\usepackage{amsfonts,amssymb,color}
\usepackage[mathscr]{eucal}
\usepackage{amsmath, amsthm}
\usepackage{mathrsfs}
\usepackage{amsbsy}
\usepackage{dsfont}
\usepackage{bbm}
\usepackage{wasysym}
\usepackage{stmaryrd}
\usepackage{url}

\input xypic
\xyoption{all}

\makeindex
\makeglossary

\begin{document}
\baselineskip = 16pt

\newcommand \ZZ {{\mathbb Z}}
\newcommand \NN {{\mathbb N}}
\newcommand \RR {{\mathbb R}}
\newcommand \PR {{\mathbb P}}
\newcommand \AF {{\mathbb A}}
\newcommand \GG {{\mathbb G}}
\newcommand \QQ {{\mathbb Q}}
\newcommand \CC {{\mathbb C}}
\newcommand \bcA {{\mathscr A}}
\newcommand \bcC {{\mathscr C}}
\newcommand \bcD {{\mathscr D}}
\newcommand \bcF {{\mathscr F}}
\newcommand \bcG {{\mathscr G}}
\newcommand \bcH {{\mathscr H}}
\newcommand \bcM {{\mathscr M}}
\newcommand \bcI {{\mathscr I}}
\newcommand \bcJ {{\mathscr J}}
\newcommand \bcK {{\mathscr K}}
\newcommand \bcL {{\mathscr L}}
\newcommand \bcO {{\mathscr O}}
\newcommand \bcP {{\mathscr P}}
\newcommand \bcQ {{\mathscr Q}}
\newcommand \bcR {{\mathscr R}}
\newcommand \bcS {{\mathscr S}}
\newcommand \bcV {{\mathscr V}}
\newcommand \bcU {{\mathscr U}}
\newcommand \bcW {{\mathscr W}}
\newcommand \bcX {{\mathscr X}}
\newcommand \bcY {{\mathscr Y}}
\newcommand \bcZ {{\mathscr Z}}
\newcommand \goa {{\mathfrak a}}
\newcommand \gob {{\mathfrak b}}
\newcommand \goc {{\mathfrak c}}
\newcommand \gom {{\mathfrak m}}
\newcommand \gon {{\mathfrak n}}
\newcommand \gop {{\mathfrak p}}
\newcommand \goq {{\mathfrak q}}
\newcommand \goQ {{\mathfrak Q}}
\newcommand \goP {{\mathfrak P}}
\newcommand \goM {{\mathfrak M}}
\newcommand \goN {{\mathfrak N}}
\newcommand \uno {{\mathbbm 1}}
\newcommand \Le {{\mathbbm L}}
\newcommand \Spec {{\rm {Spec}}}
\newcommand \Gr {{\rm {Gr}}}
\newcommand \Pic {{\rm {Pic}}}
\newcommand \Jac {{{J}}}
\newcommand \Alb {{\rm {Alb}}}
\newcommand \Corr {{Corr}}
\newcommand \Chow {{\mathscr C}}
\newcommand \Sym {{\rm {Sym}}}
\newcommand \Prym {{\rm {Prym}}}
\newcommand \cha {{\rm {char}}}
\newcommand \eff {{\rm {eff}}}
\newcommand \tr {{\rm {tr}}}
\newcommand \Tr {{\rm {Tr}}}
\newcommand \pr {{\rm {pr}}}
\newcommand \ev {{\it {ev}}}
\newcommand \cl {{\rm {cl}}}
\newcommand \interior {{\rm {Int}}}
\newcommand \sep {{\rm {sep}}}
\newcommand \td {{\rm {tdeg}}}
\newcommand \alg {{\rm {alg}}}
\newcommand \im {{\rm im}}
\newcommand \gr {{\rm {gr}}}
\newcommand \op {{\rm op}}
\newcommand \Hom {{\rm Hom}}
\newcommand \Hilb {{\rm Hilb}}
\newcommand \Sch {{\mathscr S\! }{\it ch}}
\newcommand \cHilb {{\mathscr H\! }{\it ilb}}
\newcommand \cHom {{\mathscr H\! }{\it om}}
\newcommand \colim {{{\rm colim}\, }} 
\newcommand \End {{\rm {End}}}
\newcommand \coker {{\rm {coker}}}
\newcommand \id {{\rm {id}}}
\newcommand \van {{\rm {van}}}
\newcommand \spc {{\rm {sp}}}
\newcommand \Ob {{\rm Ob}}
\newcommand \Aut {{\rm Aut}}
\newcommand \cor {{\rm {cor}}}
\newcommand \Cor {{\it {Corr}}}
\newcommand \res {{\rm {res}}}
\newcommand \red {{\rm{red}}}
\newcommand \Gal {{\rm {Gal}}}
\newcommand \PGL {{\rm {PGL}}}
\newcommand \Bl {{\rm {Bl}}}
\newcommand \Sing {{\rm {Sing}}}
\newcommand \spn {{\rm {span}}}
\newcommand \Nm {{\rm {Nm}}}
\newcommand \inv {{\rm {inv}}}
\newcommand \codim {{\rm {codim}}}
\newcommand \Div{{\rm{Div}}}
\newcommand \CH{{\rm{CH}}}
\newcommand \sg {{\Sigma }}
\newcommand \DM {{\sf DM}}
\newcommand \Gm {{{\mathbb G}_{\rm m}}}
\newcommand \tame {\rm {tame }}
\newcommand \znak {{\natural }}
\newcommand \lra {\longrightarrow}
\newcommand \hra {\hookrightarrow}
\newcommand \rra {\rightrightarrows}
\newcommand \ord {{\rm {ord}}}
\newcommand \Rat {{\mathscr Rat}}
\newcommand \rd {{\rm {red}}}
\newcommand \bSpec {{\bf {Spec}}}
\newcommand \Proj {{\rm {Proj}}}
\newcommand \pdiv {{\rm {div}}}
\newcommand \NS{{\rm{NS}}}
\newcommand \wt {\widetilde }
\newcommand \ac {\acute }
\newcommand \ch {\check }
\newcommand \ol {\overline }
\newcommand \Th {\Theta}
\newcommand \cAb {{\mathscr A\! }{\it b}}

\newenvironment{pf}{\par\noindent{\em Proof}.}{\hfill\framebox(6,6)
\par\medskip}

\newtheorem{theorem}[subsection]{Theorem}
\newtheorem{conjecture}[subsection]{Conjecture}
\newtheorem{proposition}[subsection]{Proposition}
\newtheorem{lemma}[subsection]{Lemma}
\newtheorem{remark}[subsection]{Remark}
\newtheorem{remarks}[subsection]{Remarks}
\newtheorem{definition}[subsection]{Definition}
\newtheorem{corollary}[subsection]{Corollary}
\newtheorem{example}[subsection]{Example}
\newtheorem{examples}[subsection]{examples}

\title{Representability of Chow groups of codimension three cycles}
\author{Kalyan Banerjee}

\address{Harish Chandra Research Institute, India}

\email{banerjeekalyan@hri.res.in}

\begin{abstract}
In this note we are going to prove that if we have a fibration of smooth projective varieties $X\to S$ over a surface $S$ such that $X$ is of dimension four and that the geometric generic fiber has finite dimensional motive and the first \'etale cohomology of the geometric generic fiber with respect to $\QQ_l$ coefficients is zero and the second \'etale cohomology is spanned by divisors, then $A^3(X)$ (codimension three algebraically trivial cycles modulo rational equivalence) is dominated by finitely many copies of $A_0(S)$. Meaning that there exists finitely many correspondences $\Gamma_i$ on $S\times X$, such that $\sum_i \Gamma_i$ is surjective from $\oplus A^2(S)$ to $A^3(X)$.
\end{abstract}

\maketitle

\section{Introduction}

The representability problem in the theory of algebraic cycles is a very interesting and a fundamental problem. Precisely it means the following. Let $X$ be a smooth projective algebraic variety of dimension $n$ over an algebraically closed ground field $k$ of characteristic zero. Consider the group of algebraic cycles of codimension $i$ which are algebraically trivial modulo rational equivalence. Denote this group by $A^i(X)$. Then the question is, when there exists a smooth projective curve $C$ defined over $k$ and a correspondence $\Gamma$ on $C\times X$ such that $\Gamma_*$ from $J(C)$, the Jacobian variety of $C$, to $A^i(X)$ is onto. The case when we consider $A^n(X)$, this representability question is equivalent to the fact that $A^n(X)$ is isomorphic to the albanese variety of $X$, which is also equivalent to the surjectivity of the natural map from some high degree symmetric power of $X$ to $A^n(X)$. It is a conjecture due to Bloch that when we consider a smooth projective surface $S$ with geometric genus zero then the group $A^2(S)$ is representable. On the other hand, Mumford \cite{M} proved that when the geometric genus of the surface is greater than zero then the group $A^2(S)$ is not representable. Bloch's conjecture for surfaces with geometric genus equal to zero has been proved in certain cases, for all surfaces not of general type \cite{BKL} and some examples of surfaces of general type \cite{V},\cite{VC}.

In \cite{VG}[Theorem 1] it has been proved that when we have a smooth projective threefold $X$ fibered into surfaces over a smooth projective curve $C$, such that the geometric generic fiber has finite dimensional motive, has first \'etale cohomology with $\QQ_l$ is zero and the second \'etale cohomology with $\QQ_l$ is spanned by divisors, then the group $A^2(X)$ is representable in the sense that there exists finitely many correspondences $\Gamma_i$ on $C\times X$, such that $\oplus_i \Gamma_{i*}$ from $\oplus_i J(C)$  to $A^2(X)$ is onto. Then as an application, it has been proved that the $A^2$ of a del Pezzo fibration over a smooth projective curve is representable.

In this paper our aim is to extend the result of \cite{VG} to the case when $X$ is of dimension $4$ and it is fibered into surfaces over a smooth projective surface, such that the geometric generic fiber satisfies the property as above. Then we prove that $A^3(X)$ is representable up to dimension $2$. Precisely it means that there exists finitely many correspondences $\Gamma_i$ on $S\times X$ such that $\oplus_i \Gamma_{i*}$ from $A^2(S)$ to $A^3(X)$ is onto. In other words we prove that $A^3(X)$ is representable by $A^2$ of smooth projective surfaces.
\smallskip

So the main theorem is :

\begin{theorem}
Let $X$ be a smooth projective fourfold birational to a fourfold $X'$ fibered over a surface $S$. Assume moreover that the geometric generic fiber of the fibration $X'\to S$ satisfies the following:

(i) The motive of it is finite dimensional.
(ii) First \'etale cohomology of it is trivial with respect to $\QQ_l$ coefficients.
(iii) The second \'etale cohomology is spanned by divisors on it.

Then the group $A^3(X)$ is representable up to dimension two.
\end{theorem}

\smallskip

The underlying  technique to prove the main theorem is same as in the proof of Theorem 1, \cite{VG}, but the only non-trivial step is to excise a curve from the base of the fibration and to prove that the representability of $A^3(X)$ will follow from representability of $A^3(X_U)$, where $U=S\setminus C$, that is the part we remove has representable $A^2$.

The theorem is interesting from the following view point: The representability of $A^3$ up to dimension $2$ is a birational invariant of smooth projective fourfolds that holds for rational varieties. Hence one motivation for the above mentioned theorem is to the rationality problem, where we explain the vanishing of this obstruction to rationality for smooth, projective fourfolds fibered over surfaces. In one case of interest to the rationality problem, when $X$ is a cubic fourfold, it should be noted that the representability of $A^3$ is already known by \cite{BP}[Proposition 2.7] and \cite{SV}[part 3].

\smallskip

{\small \textbf{Acknowledgements:} The author thanks Kapil Paranjape for his constant encouragement and for carefully listening about the arguments of this paper from the author. The author thanks R.Laterveer and the anonymous referee for pointing out a mistake present in  the previous version of the paper.}

{\small Throughout this text we work over an algebraically closed ground field $k$ of characteristic zero and all Chow groups are considered with $\QQ$-coefficients.}

\section{Representability up to dimension two of codimension three cycles}
\label{section2}
Let $X$ be a smooth projective variety and let $A^i(X)$ denote the algebraically trivial, codimension $i$ algebraic cycles on $X$, modulo rational equivalence. Then we say  that $A^i(X)$ is weakly representable up to dimension two if there exists finitely many curves $C_1,\cdots,C_m$ with correspondences $\Gamma_1,\cdots,\Gamma_m$ on $C_1\times X, \cdots, C_m\times X$ and finitely many surfaces $S_1,\cdots,S_n$ with correspondences $\Gamma_j'$ on $S_j\times X$, such that
$$\sum_i \Gamma_i+\sum_j \Gamma_j'$$
is surjective from $\oplus_i A^1(C_i)\oplus _j A^2(S_j)$ to $A^i(X)$. If we assume that $X$ is a fourfold, then the representability of $A^2(X)$ is a birational invariant. This is because if we blow up $X$ to $\wt{X}$, then $A^2(\wt{X})$ is isomorphic  to $A^2(X)\oplus A^1(Z)$, where $Z$ is the center of  the blow up. Since $A^1(Z)$ is dominated by $J(\Gamma)$, for some smooth projective curve $\Gamma$, this will imply  that if $A^2(X)$ is representable up to dimension two then so is $A^2(\wt{X})$. Suppose  that $X,Y$ are birational, such that $Y$ is obtained by one blow up of $X$ and then one blow down, the we have a generically finite map from $\wt{X}$ to $Y$, which gives a surjection at the level of $A^2$. So $A^2(X)$ representable up to dimension two implies the same for $A^2(\wt{X})$, hence the same for $A^2(Y)$. Changing the role of $X,Y$, we get the reverse implication.

Similarly if we consider the representability of $A^3(X)$ up to dimension two, where $X$ is a smooth projective fourfold, then it is a birational invariant in $X$. This is because if we blow up $X$ along a surface or a curve then the blow up formula gives us
$$A^3(\wt{X})=A^3(X)\oplus A^2(S)\oplus A^1(S)$$
or
$$A^3(\wt{X})=A^3(X)\oplus A^1(C)$$
where $S$ or $C$ is the center of the blow up. So if we blow up for many times we are only adding $A^2$ of a surface or $A^1$ of a curve, so the representability up to dimension two remains.

So our main theorem in this section is the following.

\begin{theorem}
\label{theorem1}
Let $X$ be a smooth projective fourfold birational to a fourfold $X'$ fibered over a surface $S$. Assume moreover that the geometric generic fiber of the fibration $X'\to S$ satisfies the following:

(i) The motive of it is finite dimensional.
(ii) First \'etale cohomology of it is trivial with respect to $\QQ_l$ coefficients.
(iii) The second \'etale cohomology with respect to $\QQ_l$ coefficients, is spanned by divisors on it.

Then the group $A^3(X)$ is representable up to dimension two.
\end{theorem}

\begin{proof}
Let us assume from the very beginning that the fourfold $X$ is equipped with a fibration to a smooth projective surface $S$. That is we have a fibration $X\to S$. Let $\eta=\Spec(k(S))$, and $\bar{\eta}=\Spec(\overline{k(S)})$. Let $b_2$ be the dimension of $H^2_{\'et}(X_{\bar{\eta}},\QQ_l)$ and let by our assumption $D_1,\cdots,D_{b_2}$ be the divisors on $X_{\bar{\eta}}$, generating the second \'etale cohomology group $H^2_{\'et}(X_{\bar\eta},\QQ_l)$. Let us consider a finite extension $L$ of $k(S)$, inside its algebraic closure such that $D_1,\cdots,D_{b_2}$ are defined over $L$. That is we consider a smooth projective curve $S'$ mapping finitely onto $S$ with function field $L$, such that $X'=X\times _S S'\to X$ is of finite degree and $D_1,\cdots,D_{b_2}$ are defined over the generic point of $S'$. Since $X'\to X$ is finite we can work with this divisors which are actually defined over the generic point of $S'$.

Now we need the lemma.

\begin{lemma}
Let $X$ be a smooth projective fourfold over a field $k$ and let $A^3(X)=V\oplus W$, where $V$ is a finite dimensional $\QQ$ vector space. Then $A^3(X)$ is representable if and only if there exists finitely many smooth curves and surfaces $C_1,\cdots,C_m,S_1,\cdots,S_n$, and correspondences $\Gamma_i$ on $C_i\times X$, and $\Gamma_j'$ on $S_j\times X$ such that the homomorphism $\sum_i \Gamma_i+ \sum_j \Gamma_j'$ from $\oplus_i A^1(C_i)\oplus_j A^2(S_j)$ to $A^3(X)$ is surjective onto $W$.
\end{lemma}

\begin{proof}
Let $v_1,\cdots,v_n$ be a basis for $V$. For each $v_j$ let $Z_j$ be the algebraical cycle representing it. Since $Z_j$ is algebraically equivalent to zero, we have a smooth projective curve $C_j$ and a correspondence $\Gamma_j$ such that $\Gamma_{j_*}(x_j)$ equals $Z_j$, where $x_j$ is a point on $J(C_j)$. Therefore the homomorphism $\sum_j\Gamma_{j*}$ is covering the space $V$ and it has domain $\oplus J(C_j)$. So to prove that $A^3(X)$ is representable it is enough to prove the representability of $W$. So we need to find some smooth curves and surfaces satisfying the assumption that the sum of algebraically trivial zero cycles on these curves and surfaces cover $W$.
\end{proof}
\medskip

step{2}:

\smallskip

Let $\{p_1,\cdots,p_m\}$ be a finite set of closed points on $S$. Let $U$ be the complement of this finite set. Let $Y= f^{-1}(U)$. Then by the localization exact sequence we have that
$$\oplus_j\CH^2(X_{p_j})\to \CH^3(X)\to \CH^3(Y)\to 0$$
so  the $\QQ$ vector space $\CH^3(X)$ splits as $\CH^2(Y)\oplus I$
where $I$ is the image of the pushforward from $\oplus_j \CH^2(X_{p_j})$ to $\CH^3(X)$. It is also true that the map from $A^3(X)$ to $A^3(Y)$ is surjective, where $A^3$ denote the algebraically trivial one-cycles modulo rational equivalence. So we have a splitting
$$A^3(X)=A^3(Y)\oplus J$$
where $J$ is the intersection of $I$ and $A^3(X)$. Let for $X_{p_j}$, $\wt{X_{p_j}}$ is the resolution of singularity of it. Then we have that $J$ is covered by two subspaces, one is the direct sum of $A^2(X_{p_j})$, which is covered by direct sums of the $A^2$'s of the irreducible components of $\wt{X_{p_j}}$, the other is a finite dimensional subspace, coming from the Neron severi group of the irreducible components of the resolutions of $\wt{X_{p_j}}$. So by the previous lemma it is sufficient to prove that $A^3(Y)$ is representable up to dimension two to prove the representability of the group $A^3(X)$.

\medskip

step 3:

\medskip

Let $C$ be a projective curve inside $S$, and we excise $C$ from $S$. Let $Y$ be the complement of $X_C=X\times _S C$ in $X$. Then we prove that the representability of $A^3(X)$, follows from the representability of $A^3(Y)$. For that we consider the localisation exact sequence given by
$$\CH^2(X_C)\to \CH^3(X)\to \CH^3(Y)\to 0\;.$$
Then we have $\CH^3(X)=\CH^3(Y)\oplus I$, where $I$ is the image of $\CH^2(X_C)$ in $\CH^3(X)$. Considering the subgroup of algebraically trivial cycles we get that
$$A^3(X)=A^3(Y)\oplus J$$
where $J$ is the intersection of $I$ with  the image of $A^3(X)$. Then $J$ is a sum of two $\QQ$-vector spaces. One is the image of $A^2(X_C)$ and the other is a finite dimensional subspace corresponding to the Neron-Severi group of $X_C$. Then by step one if we have $A^2(X_C)$ is representable then we have the representability of $J$. But the representability of $A^2(X_C)$  follows from \cite{VG}[Theorem 1]. Because according to our assumption the geometric generic fiber of $X\to S$ has finite dimensional motive and base change of finite dimensional motive is finite dimensional. Therefore the geometric generic fiber of $X_C\to C$ has finite dimensional motive. Also the first and second etale cohomology of the geometric generic fiber of $X_C\to C$ satisfies the assumption of \cite{VG}[Theorem 1], because the geometric generic fiber of $X\to S$ satisfies the similar properties. Therefore we have the representability of $A^3(X)$ follows from that of $A^3(Y)$. So we can say that to prove representability of $A^3(X)$ it is sufficient to remove a finitely many curves from the base, and look for the representability of the $A^3(Y)$, where $Y$ is he complement of $\cup_i X_{C_i}$.

step 4:

\smallskip

Suppose that $X_{\eta}$ is defined over a finite extension $L$ of $k(S)$ inside $\overline{k(S)}$. Then let $S'$ be a smooth projective surface with function field $L$, and mapping finitely onto $S$. Now over $S'$ we have a rational point of the variety $X_{\eta}'=X_{\eta}\times_{k(S)} S'$. This rational point induces a section of the map $Y\to U$, over some $U'$ Zariski open inside $U$.  Now $U'$ maps isomorphically onto its image in $Y$. So we have to remove a curve from $U$ to obtain $U'$. Since the representability remains unchanged by this process, we can assume without loss of generality that the section is defined everywhere on $U$.  So without loss of generality we can assume that $Y\to U$ has a section. Let $E$ be the image of this section. Then $E.E$ has codimension $4$ in $Y$, so it's support is contained in finitely many fibers. So we can cut down those finitely many fibers. Then we can prove that  $\pi_0=E\times _U Y,  \pi_4=Y\times _U E$ are pairwise orthogonal \cite{VG}[page 332, reduction 4]. Hence we have the projector
$$\pi_2=\Delta_{Y/U}-\pi_0-\pi_4\;.$$

Let $M^2(Y/U)$ be the relative motive defined by $\pi_2$. Then we have the decomposition
$$M(Y/U)=\mathbb {1}_U\oplus M^2(Y/U)\oplus \mathbb {L}^2_U$$

Now we know that $M(X_{\bar\eta})$ is finite dimensional, which means at the level of Chow groups that there exists some correspondence $p,q$ on $X_{\bar\eta}$ such that $d_{\Sym}\circ p^n$ is rationally trivial and $d_{alt}\circ q^n$ is rationally trivial. Let $L$ be the minimal field of definition of $p,q$, then taking a finite extension $S'$ over $S$, with function field $L$, we have $M(Y_{\eta})$ is finite dimensional over $\eta$ itself. On the other hand since $\CH^2(Y_{\eta}\times Y_{\eta})$ is the colimit of the groups $CH^2(Y_U\times_U Y_U)$, we have that the motive $M(Y/U)$ is finite dimensional for some open set $U$ in $S$. Then we shrink our $U$ to this $U$ by taking intersection.

Now the finite dimensionality of $M(Y/U)$ implies $M^2(Y/U)$ is finite dimensional. One can show more, that is $M^2(Y/U)$ is evenly finite dimensional of dimension $b_2$. This follows from the computation of \cite{VG}[\textit{Main computations}, page 333].

Now let $D_1,\cdots,D_{b_2}$ be the divisors defined over $\eta$ and they generate the cohomology group $H^2(Y_{\eta},\QQ_l)$. According to \cite{VG}[page 334],\cite{GP}[theorem 2.14] we have

$$\rho_{\eta}=(\pi_2)_{\eta}-\sum_{i=1}^{b_2}[D_i\times_{\eta}D_i']$$
is homologically trivial. Then there exists some $n$ such that $\rho_{\eta}^n=0$, in the associative ring $End(M^2(Y_{\eta}))$, by Kimura's nilpotency theorem \cite{KI}[proposition 7.2].

Let $W_i,W_i'$ are spreads  of the above divisors over $U$, they may be non-unique but we choose and fix one spread. Consider the cycles
$$W_i\times _U W_i'$$
in $Corr^0_U(Y\times_U Y)$ and set
$$\rho=\pi_2-\sum_{i=1}^{b_2} [W_i\times_U W_i']$$
then $\rho$ maps to $\rho_{\eta}$ under the base change functor from the category of relative Chow motives over $U$ to the category of Chow motives over $\eta$. Let us consider an endomorphism $\omega$ of $M^2(Y/U)$. Then under the above functor trace of $\omega\circ \rho$ is mapped to trace of $\omega_{\eta}\circ \rho_{\eta}$ \cite{VG}[page 334], \cite{DM}[page 116], which is zero because $\rho_{\eta}$ is homologically trivial. The base change functor defines an isomorphism from $End(\mathbb 1_U)$ to $End(\mathbb 1_{\eta})$. Therefore trace of $\omega\circ \rho=0$ for any $\omega$, so $\rho $ is numerically trivial, therefore $\rho^n=0$ by \cite{VG}[Proposition 2], \cite{KI}[7.5],\cite{AK}[9.1.14].

Let $\bar{W_i}$ be the Zariski closure of $W_i$ in $X$ and consider
$$\theta_i=\Gamma^t_f.[S\times \bar{W_i}]$$
it is a codimension $3$ cycle on $S\times X$. The cycle $\Gamma^t_f$ is the transpose of the graph of the map $f:X
\to S$. Consider the homomorphism $\theta_{i*}$ from $\CH^2(S)$ to $\CH^3(X)$. Let us compute $\theta_{i*}$.

$$\theta_{i*}(a)=p_{X*}(p_S^*(a).\theta_i)$$
which is equal to
$$\theta_{i*}(a)=p_{X*}(p_S^*(a).\Gamma^t_f.[S\times \bar{W_i}])$$
on the other hand we have $p_S^*(a).\Gamma^t_f=p_S^*(a).\tau_*([X])$, where $\tau $ is the map $x\mapsto (f(x),x)$.
We have $f^*(a)=\tau^*p_S^*(a)=\tau^*p_S^*(a).[X]$
therefore $\tau_*f^*(a)=\tau_*(\tau^*p_S^*(a).[X])$, which by projection formula is $p_S^*(a).\tau_*(X)=p_S^*(a).\Gamma^t_f$. Putting this in the above expression of $\theta_{i*}$ we have
$$\theta_{i*}(a)=p_{X*}(\tau_*f^*(a).[C\times \bar{W_i}])$$
$$=p_{X*}(\tau_*f^*(a).p_X^*([\bar{W_i}]))=p_{X*}\tau_*f^*(a).[\bar{W_i}]=f^*(a).[\bar{W_i}]\;.$$
So this computation provides the description of the homomorphism $\theta_{i*}$ in the non-compact case when we consider it from $\CH^2(U)$ to $\CH^3(Y)$. It is immediate that the homomorphisms $\theta_{i*}$'s are compatible in compact and non-compact cases. Since the homomorphism $\theta_{i*}$ in  the non-compact case respects algebraic equivalence we have the compatibility at the level of algebraically trivial cycles modulo rational equivalence. So summarising we have a commutative diagram as follows.

$$
  \diagram
\sum_{i=1}^{b_2} A^2(S)\ar[dd]_-{} \ar[rr]^-{\theta_*} & & A^3(X) \ar[dd]^-{} \\ \\
\sum_{i=1}^{b_2}  A^2(S) \ar[rr]^-{\theta_*} & & A^3(Y)
  \enddiagram
  $$
Chasing the above diagram and assuming that the bottom $\theta_*$ is surjective we have that the top $\theta_*$ has image equal to $A^3(X)$ modulo $A^2(X_C)$, where $C$ is the complement of $U$ in $S$ and $X_C=f^{-1}(C)$. Since $A^2(X_C)$ is finite dimensional it is enough to prove that $\theta_*$ at the bottom is onto to prove the representability of $A^3(X)$ up to dimension $2$.

Let $y$ belongs to $\CH^3(Y)$, then considering the relative correspondence $\Delta_{Y/U}$, we get that
$$y=\Delta_{Y/U*}(y)=\pi_{0*}(y)+\pi_{2*}(y)+\pi_{4*}(y)\;.$$

Now $\pi_{0*}(y)$ is equal to $p_{2*}(p_1^*(y).\pi_0)$ which is equal to $p_{2*}(p_1^*(y).p_1^*(E))=p_2^*p_1^*(y.E)=f^*f_*(y.E)=0$ as the codimension of $y.E$ is five. So we have $\pi_{0*}(y)=0$. Also we have $f_*(y)=0$.

Next we compute,
$$\pi_{4*}(y)=p_{2*}(p_1^*(y).\pi_4)$$
$$=p_{2*}(y\times_U Y.Y\times_U E)=p_{2*}(y\times _U E)$$
$$=f_*(y)\times_U E=0\;.$$

So we have that $y=(\pi_{2*})(y)\;.$
Putting $\pi_2$ equal to $\sum_i [W_i\times_U W_i']+\rho$
we get that $y=\pi_{2*}(y)=\sum_i [W_i\times_U W_i']_*(y)+\rho_*(y) $. Let $Z_j$'s are curves representing the class of $y$, then
$$[W_i\times_U W_i']_*(Z_j)=p_{2*}([Z_j\times_U Y].[W_i\times_U W_i'])$$
$$=p_{2*}([Z_j].[W_i]\times_U [Y].[W_i'])=p_{2*}([Z_j].[W_i]\times [W_i'])$$
by linearity we have
$$[W_i\times_U W_i'](y)=p_{2*}(y.[W_i]\times_U [W_i'])$$
since $y$ is of codimension $3$ and $W_i$ is of codimension $1$, we have $y.W_i$ is a zero cycle on $Y$.
Observe that
$$[W_i\times_U W_i']_*(y)=p_{2*}(y.W_i\times_U W_i')$$
$$=p_{2*}(p_1^*(y.W_i).p_2^*(W_i'))=p_{2*}p_1^*(y.W_i).W_i'$$
$$=f^*f_*(y.W_i).W_i'=f^*(a_i).W_i'=\theta_{i*}(a_i)$$
where $a_i=f_*(y.W_i)$. Since $y$ belongs to $A^3(Y)$, we have hat $a_i$ is in $A^2(U)$. Then we get that
$$\sum_i[W_i\times_U W_i']_*(y)=\sum_i \theta_{i*}(a_i)=\theta_*(e_1)$$
where $c_1=(a_1,\cdots,a_{b_2})$ in $\oplus_i A^2(S)$. So we have
$$\rho_*(y)=\theta_*(c_1)+y$$ applying $\rho$ n-times we have that
$$\rho^n_*(y)=0=\theta_*(c_n)+ny$$
so we have $y=-1/n \theta_*(c_n)$, hence $\theta_*$ is surjective.

\end{proof}

\begin{remark}
It is  interesting to note that one of the conditions in Theorem \ref{theorem1} is the motivic finite dimensionality of the geometric generic fiber $X_{\bar\eta}$. Suppose that the ground field is $\CC$. Consider $X_{\bar\eta}$ over $\overline{\CC(S)}$. Then the motivic finite dimensionality is also equivalent to Bloch's conjecture for the geometric generic fiber if the geometric genus of the fiber is zero, see \cite{GP1}[Theorem 27]. Recall that, universal triviality of the Chow group of zero cycles on the surface $X_{\bar \eta}$ defined over the algebraically closed field $\overline {k(S)}$ is
$$\CH_0({X_{\bar\eta}}_{(\overline{k(S)})(X_{\bar\eta})})\cong \ZZ\;.$$
This is equivalent, \cite{ACP}[proof of Lemma 1.3], to the integral decomposition of the diagonal $\Delta_{X_{\bar\eta}}\subset X_{\bar \eta}\times X_{\bar \eta}$, which says that
$$[\Delta_{X_{\bar\eta}}]=[Z_1]+[Z_2]$$
in $\CH^2(X_{\bar\eta}\times X_{\bar\eta})$,
where $[.]$ denote the cycle class modulo rational equivalence in the Chow group, $Z_1$ is supported on $D\times X_{\bar\eta}$, $D\subsetneq X_{\bar \eta}$ and $Z_2=X_{\bar\eta}\times x$ for a $\overline{k(S)}$-point $x$ on $X_{\bar\eta}$.                                                                                                            Let us now consider the case $k=\CC$. Suppose that the geometric genus of the surface ${X_{\bar\eta}}_{\CC}$, obtained under the extension of scalars from $\overline{\CC(S)}$ to $\CC$, is zero. Suppose also that the torsion in the Neron-Severi group of $X_{\bar\eta}$ is trivial. Since torsion in the Neron-Severi group remains unchanged by the extension of scalars from $\overline{\CC(S)}$ to $\CC$ (this follows from rigidity of unramified cohomology groups as studied in \cite{C}[Section 4.4]) we have, the torsion in the Neron-Severi group of ${X_{\bar\eta}}_{\CC}$ is trivial. Now motivic finite dimensionality of $X_{\bar\eta}$ implies motivic finite dimensionality of ${X_{\bar\eta}}_{\CC}$. This further implies the Bloch's conjecture on ${X_{\bar\eta}}_{\CC}$ with the assumption on the geometric genus of ${X_{\bar\eta}}_{\CC}$. The Bloch's conjecture on ${X_{\bar\eta}}_{\CC}$ with the triviality of the torsion subgroup in the Neron-Severi group of ${X_{\bar\eta}}_{\CC}$ implies the universal triviality of the Chow group of zero cycles on $X_{\bar \eta}$, see \cite{ACP}[Proposition 1.9 + Corollary 1.10], \cite{BS}[Remark 2, page 1252], \cite{GG}[Corollary 8], \cite{Vo}[Corollary 2.2].

 On the other hand, integral decomposition of the diagonal, gives Bloch's conjecture on $X_{\bar\eta}$ by Chow moving lemma (also the geometric genus of $X_{\bar\eta}$ is zero in this case by the Proposition 1.8 in \cite{ACP}) and hence we obtain the finite dimensionality of the motive $M(X_{\bar\eta})$. Therefore for surfaces $X_{\bar\eta}$ over $k=\overline{\CC(S)}$, with ${X_{\bar\eta}}_{\CC}$ being of geometric genus zero, and $\NS(X_{\bar\eta})_{Tors}=\{0\}$ (or equivalently $\NS({X_{\bar\eta}}_{\CC})=\{0\}$), the motivic finite dimensionality condition in  Theorem \ref{theorem1} can be replaced by the integral decomposition of the diagonal of $X_{\bar \eta}$ or the universal triviality of $\CH_0$ of $X_{\bar\eta}$.

Now, let $X$ be a smooth projective variety defined over the field of complex numbers with universally $\QQ$-trivial Chow group of zero cycles, that is
$$\CH_0({X}_{\CC(X)})\otimes_{\ZZ}\QQ\cong \QQ\;.$$
It follows by \cite{BS}[Proposition 1] that the cycle class of the diagonal $\Delta_X$ in $\CH^2(X\times X)$ admits a rational decomposition, that is
$$N[\Delta_X]=[Z_1]+[Z_2]$$
in $\CH^2(X\times X)$, where $[.]$ denote the class of a cycle modulo rational equivalence in the Chow group. Here $Z_1$ is supported on $D\times X$, $D\subsetneq X$ and $Z_2$ is $m(X\times x)$, for a $\CC$-point $x$ on $X$ and $m$ an integer. Let $N_X$ be the smallest positive integer for which the above mentioned decomposition of the diagonal holds. This integer $N_X$ is defined as the torsion order of the variety $X$.

Let $X\to S$ be a smooth, projective fourfold fibred into surfaces over the surface $S$ as before. Suppose that $\CH_0$ of $X_{\bar\eta}$ is universally $\QQ$-trivial. Now by the main theorem of \cite{K} [Proposition 3.3 and Corollary 6.4], \cite{KS}[Remark 3.1.5 (3)], it follows that the torsion order of the surface $X_{\bar\eta}$ is the exponent of the torsion subgroup in the Neron-Severi group of $X_{\bar\eta}$. Suppose that the torsion subgroup in the Neron-Severi group of $X_{\bar\eta}$ is trivial, then we have the universal ($\ZZ$)-triviality of the Chow group of zero cycles of $X_{\bar\eta}$. This fact is well-known and proven in \cite{ACP}[Proposition 1.8 and 1.9]. Therefore universal $\QQ$-triviality of $\CH_0$ of $X_{\bar\eta}$ with the information that the torsion subgroup of the Neron-Severi group of $X_{\bar\eta}$ is trivial give integral decomposition of the diagonal, which further implies, by the above discussion, the motivic finite dimensionality of $X_{\bar\eta}$. Hence the assumption of universal $\QQ$-triviality along with torsion free Neron-Severi group of $X_{\bar\eta}$ is stronger than the motivic finite dimensionality assumption for the geometric generic fiber in the main theorem \ref{theorem1}. Also note that in the case of universal $\QQ$-triviality of $\CH_0$ of $X_{\bar\eta}$, the geometric genus of the surface $X_{\bar\eta}$ is zero by the result of Bloch-Srinivas, \cite{BS}.
\end{remark}

\begin{example}
Let $X$ be a smooth, projective fourfold over $\CC$ fibered into del Pezzo surfaces over a smooth, projective surface, then the geometric generic fiber of this fibration satisfies the conditions of the Theorem \ref{theorem1}. Therefore $A^3(X)$ is representable up to dimension $2$. Such examples have been studied in \cite{AHTV}, where a generic cubic fourfold containing a sextic, elliptic, ruled surface is shown to be birational to a del Pezzo surface fibration over the projective plane. General del Pezzo fibrations are studied in detail in \cite{Ku}. Also it is to be mentioned that representability of $A^3$ up to dimension $2$ is known for all cubic fourfolds.

Other examples of fourfolds fibered in del Pezzo surfaces are quadric surface bundles over surfaces and involution surface bundles over surfaces. The representability of $A^3$ up to dimension $2$ holds for these examples as the geometric generic fibers of the fibrations mentioned above are del Pezzo surfaces and satisfy the condition of Theorem \ref{theorem1}. The first family of examples are studied in \cite{HPT} and the authors exhibit families of this type, having both rational and non-stably rational fibers and further showing that stable rationality for these families are non-deformable. The second family of examples are studied in \cite{KT}, \cite{KT1}. In both of these two types of examples, the authors prove that families of such fourfolds have their very general fiber (hence the geometric generic fiber) not stably rational, though the representability of $A^3$ up to dimension $2$ holds for the geometric generic fibers.

Consider a smooth, projective fourfold $X$ over $\CC$ fibered into Barlow surfaces over a smooth, projective surface, then we have the criterion of Theorem \ref{theorem1} satisfied, as Bloch's conjecture is true for Barlow surfaces \cite{V}, the motive of the geometric generic fiber of this fibration, is finite dimensional. Therefore we have $A^3(X)$ representable up to dimension $2$. Examples of such fibration can be constructed  from the universal determinental Barlow surface over the moduli space of determinental Barlow surfaces which is two dimensional. This universal family of Barlow surfaces can be constructed as a quotient of a family of certain Catanese surfaces admitting an involution. For such examples please see \cite{S}, \cite{V}[Introduction, discussion after theorem 0.6].

\end{example}


\begin{thebibliography}{AAAAAAA}
\bibitem[AHTV]{AHTV} N.Addington, B.Hassett, Y.Tschinkel, A.Varilly-Alvarado, {\em Cubic fourfolds fibered into sextic del pezzo surfaces}, https://arxiv.org/pdf/1606.05321.pdf, 2016.
\bibitem[AK]{AK}Y.Andre, B.Kahn, (with appendix by P.O'Sullivan) {\em Nilpotence, radicaux et structure monoidales}, Rend. Sem. Math. Univ. Padova 108(2002) 107-291.
\bibitem[ACP]{ACP} A.Auel, J.Colliot-Thelene, R. Parimala, {\em Universal unramified cohomology of cubic fourfolds containing a plane}, Brauer groups and obstruction problems, A.Auel, B.Hassett, A.Varilly-Alvarado, B.Viray (editors), Progress in Mathematics, Volume 320, Birkhauser.
\bibitem[BKL]{BKL}S.Bloch, A.Kas, D.Lieberman, {\em Zero cycles on surfaces with $p_g=0$}, Compositio Mathematicae, tome 33, no 2(1976), page 135-145.
\bibitem[BS]{BS} S. Bloch, V.Srinivas, {\em Remarks on correspondences and algebraic cycles}, American Journal of Mathematics, Volume 105, No. 5, 1983, 1235-1253.
\bibitem[BP]{BP} M. Bolognesi and C. Pedrini, {\em Transcendental motive of a cubic fourfold}, https://arxiv.org/pdf/1710.05753.pdf
\bibitem[C]{C} J.Colliot-Thelene, {\em Birational invariants, purity and the Gersten conjecture}, Lectures at the 1992 AMS Summer School, Santa Barbara, California.


\bibitem[DM]{DM}P.Deligne, J.Milne, {\em Tannakian categories}, In Hodge cycles and Shimura varieties, Lecture notes in Math. 900, Springer, 1982, 101-208
\bibitem [G]{VG} V.Guletskii, {\em On the continuous part of codimension two algebraic cycles on threefolds over a field}, Math Sbornik, Volume 200, number 3.
\bibitem[GG]{GG} V.Guletskii, S.Gorchinsky, {\em Transcendence degree of zero cycles and structure of Chow motives}, Central European Journal of Mathematics, 10, 2012, 559-568.
\bibitem[GP]{GP} V.Guletskii, C.Pedrini, {\em The Chow motive of the Godeaux surface}, Algebraic Geometry, a volume in the memory of Paolo Francia, Walter De Gruyter, 2002, 179-195.
\bibitem[GP1]{GP1} V.Guletskii, C.Pedrini, {\em Finite dimensional motives and conjectures of Beilinson and Murre}, Special issue in the honor of Hyman Bass on his seventieth birthday, part III, K-Theory, 30, 2003, no. 3, 243-263.
\bibitem[HPT]{HPT} B.Hassett, A. Pirutka, Y.Tschinkel, {\em Stable rationality of quadric surface bundles over surfaces}, Acta Math., Volume 220, Number 2 (2018), 341-365.
\bibitem[K]{K} B.Kahn, {\em Torsion order of smooth projective surfaces, (with an appendix by J.L.Colliot-Thelene)}, Comment. Math. Helv., 92, 2017, 839-857.
\bibitem[KS]{KS} B.Kahn, R.Sujatha, {\em Pure birational motives I}, Annals of K-theory, 1, 2016, 379-440.
\bibitem[KI]{KI}S.Kimura, {\em Chow groups are finite dimensional in some sense}, Math.Ann., 331 (2005), no.1, 173-201.
\bibitem[KT]{KT} A.Kresch, Y.Tschinkel, {\em Involution surface bundles over surfaces}, Math. Zeitschrift, 296, 1081-1100, 2020.
\bibitem[KT1]{KT1} A.Kresch, Y. Tschinkel, {\em Brauer groups of involution surface bunles}, online preprint.
\bibitem[Ku]{Ku} A.Kuznetsov, {\em Derived categories of families of Sextic del Pezzo surfaces}, IMRN, https://doi.org/10.1093/imrn/rnz081, 2019.



\bibitem[M]{M} D.Mumford, {\em Rational equivalence for $0$-cycles on surfaces.}, J.Math Kyoto Univ. 9, 1968, 195-204
\bibitem[R]{R} A.A.Roitman, {\em The torsion of the group of $0$-cycles modulo rational equivalence}, Ann. of Math. (2) 111 (1980), no. 3, 553-569.


\bibitem[SV]{SV} M.Shen, C.Vial, {\em Fourier transform of certain hyperkahler fourfolds}, Memoirs of the AMS, 2015.
\bibitem[S]{S}P.Supino, {\em A note on Campedelli surfaces}, Geometrie Dedicatae, 71, 1998, no.1, 19-31.

\bibitem[V]{V}C.Voisin, {\em Bloch's conjecture for Catanese and Barlow surfaces}, Journal of Differential Geometry, 2014, no.1, 149-175.
\bibitem[VC]{VC} C.Voisin, {\em Sur les zero cycles de certaines hypersurfaces munies d'un automorphisme}, Ann. Scuola Norm. Sup. Pisa Cl. Sci., (4), 19, 1992, no.4, 473-492.
\bibitem[Vo]{Vo} C.Voisin, {\em Universal triviality of $\CH_0$ of cubic hypersurfaces}, Journal of European Math. Society, 19, 1619-1653, 2017.
\end{thebibliography}
\end{document}